\documentclass[11pt]{amsart}

\newcommand*{\LatexDef}{.}
\usepackage{kpfonts}

\usepackage{amsthm}

\usepackage{amsmath}

\usepackage{amssymb}

\usepackage{dsfont} %for blackboard 1

\usepackage{thmtools}

\usepackage{cancel}
\usepackage{verbatim}

\usepackage[shortlabels]{enumitem}
\setlist[enumerate]{leftmargin=*,font=\upshape,align=parleft,label=(\alph*)}
\setlist[itemize]{leftmargin=*,labelwidth=*}
\usepackage{graphicx}
\usepackage{float}

\usepackage{color}

\usepackage{xspace}

\usepackage{calrsfs}

\DeclareSymbolFont{defaultmathcal}{OMS}{zplm}{m}{n}
\DeclareSymbolFontAlphabet{\mathcal}{defaultmathcal}

\DeclareSymbolFont{handwritten}{OMS}{rsfs}{m}{n}
\DeclareSymbolFontAlphabet{\handcal}{handwritten}
\usepackage[alphabetic, nobysame]{amsrefs}

\definecolor{darkgreen}{RGB}{54,124,50}
\usepackage{hyperref}
\hypersetup{
	pdfstartview={XYZ null null 1.00}, 
	pdfpagemode=UseNone, 
	colorlinks,
	breaklinks,
	linkcolor=blue,
	urlcolor=blue, 
	anchorcolor=blue,
	citecolor=darkgreen,
}

\usepackage[capitalise]{cleveref}

\usepackage{xparse} %for defining Roman numeral command \RN
\usepackage{geometry}
\usepackage{microtype}

\ExplSyntaxOn
\NewDocumentCommand{\RN}{m}
{
	\textup{ \int_to_Roman:n { #1 } }
}
\NewDocumentCommand{\rn}{m}
{
	\textup{ \int_to_roman:n { #1 } }
}
\ExplSyntaxOff

\newcommand{\N}{\mathbb{N}}
\newcommand{\Z}{\mathbb{Z}}

\newcommand{\R}{\mathbb{R}}

\newcommand{\0}{\mathbb{\emptyset}}
\newcommand{\w}{\infty}

\newcommand{\1}{\mathds{1}} %Already defined in the package xy

\newcommand\de{\delta}
\newcommand{\e}{\varepsilon}

\newcommand\si{\sigma}

\renewcommand{\phi}{\varphi}

\DeclareRobustCommand{\rchi}{{\mathpalette\irchi\relax}}
\newcommand{\irchi}[2]{\raisebox{\depth}{$#1\cchi$}}
\let\cchi\chi
\let\chi\rchi

\newcommand{\dom}{\operatorname{dom}}

\makeatletter
\def\moverlay{\mathpalette\mov@rlay}

\def\mov@rlay#1#2{\leavevmode\vtop{%
		\baselineskip\z@skip \lineskiplimit-\maxdimen
		\ialign{\hfil$\m@th#1##$\hfil\cr#2\crcr}}}

\newcommand{\charfusion}[3][\mathord]{
	#1{\ifx#1\mathop\vphantom{#2}\fi
		\mathpalette\mov@rlay{#2\cr#3}
	}
	\ifx#1\mathop\expandafter\displaylimits\fi}
\makeatother

\renewcommand{\le}{\leqslant}

\renewcommand{\ge}{\geqslant}

\newcommand{\shortiff}{\Leftrightarrow}

\newcommand*{\defeq}{\mathrel{\vcenter{\baselineskip0.5ex \lineskiplimit0pt \hbox{\scriptsize.}\hbox{\scriptsize.}}}=}
\newcommand*{\eqdef}{=\mathrel{\vcenter{\baselineskip0.5ex \lineskiplimit0pt \hbox{\scriptsize.}\hbox{\scriptsize.}}}}

\newcommand*{\defequiv}{\mathrel{\vcenter{\baselineskip0.5ex \lineskiplimit0pt
			\hbox{\scriptsize.}\hbox{\scriptsize.}}}\shortiff}

\DeclareMathSymbol{\lqm}{\mathord}{operators}{``}
\DeclareMathSymbol{\rqm}{\mathord}{operators}{`'}

\newcommand{\set}[1]{\left\{ #1 \right\}}
\newcommand{\setbig}[1]{\big\{ #1 \big\}}

\newcommand{\clintbig}[1]{\big[ #1 \big]}

\theoremstyle{plain}

\newtheorem{theorem}[equation]{Theorem}
\newtheorem*{theorem*}{Theorem}

\makeatletter
\def\@empty{}
\def\ifemptycredit#1{%
	\def\tmp{#1}%
	\ifx\tmp\@empty%
	\else%
	{~(#1)}%
	\fi%
}
\makeatother

\newenvironment{namedthm*}[2][]{
	\par\medskip\noindent \textbf{#2}\ifemptycredit{#1}\textbf{.}\itshape\xspace
}{}
\crefname{prop}{Proposition}{Propositions}

\newtheorem*{propo*}{Proposition}

\crefname{property}{Property}{Properties}

\newtheorem*{property*}{Property}

\newtheorem{lemma}[equation]{Lemma}
\newtheorem*{lemma*}{Lemma}

\crefname{claimlemma}{Claim}{Claims}

\crefname{cor}{Corollary}{Corollaries}

\newtheorem*{cor*}{Corollary}

\crefname{obs}{Observation}{Observations}

\newtheorem*{obs*}{Observation}

\crefname{obss}{Observations}{Observations}

\newtheorem{obss*}{Observations}

\crefname{fact}{Fact}{Facts}

\newtheorem*{fact*}{Fact}

\theoremstyle{definition}

\crefname{defn}{Definition}{Definitions}

\newtheorem*{defn*}{Definition}

\newenvironment{defn**}[1][]{\par\medskip\noindent \textbf{Definition\xspace#1.}\xspace}{}

\crefname{question}{Question}{Questions}

\newtheorem*{question*}{Question}

\crefname{conj}{Conjecture}{Conjectures}

\newtheorem*{conj*}{Conjecture}

\crefname{example}{Example}{Examples}

\newtheorem*{example*}{Example}

\crefname{examples.plain}{Examples}{Examples}
\newtheorem{examples.plain}[equation]{Examples}
\newtheorem*{examples.plain*}{Examples}

\theoremstyle{remark}

\crefname{remark}{Remark}{Remarks}
\newtheorem{remark}[equation]{Remark}
\newtheorem*{remark*}{Remark}

\newenvironment{remarklike*}[2][]{\par\medskip\noindent \textit{#2}#1\textbf{.}\rmfamily\xspace}{\smallskip}

\crefname{claim+}{Claim}{Claims}
\newtheorem{claim+}[equation]{Claim}

\crefname{claim}{Claim}{Claims}

\newtheorem*{claim*}{Claim}

\crefname{subclaim}{Subclaim}{Subclaims}

\newtheorem*{subclaim*}{Subclaim}

\newenvironment{case*}[1]{\smallskip\par\noindent \textit{Case}:~#1.\rmfamily}{}

\crefname{notation}{Notation}{Notations}

\newtheorem*{notation*}{Notation}

\newtheorem*{terminology*}{Terminology}

\crefname{convention}{Convention}{Conventions}

\newtheorem*{convention*}{Convention}
\newtheorem*{conventions*}{Conventions}

\crefname{spec}{Speculation}{Speculations}

\newtheorem*{spec*}{Speculation}

\crefname{caution}{Caution}{Cautions}

\newtheorem*{caution*}{Caution}

\crefname{hypothesis}{Hypothesis}{Hypotheses}

\newtheorem*{hypothesis*}{Hypothesis}

\crefname{assumption}{Assumption}{Assumptions}

\newtheorem*{assumption*}{Assumption}

\newcommand{\fntsz}[1][11]{\fontsize{#1}{#1}\selectfont}

\crefname{examples}{Examples}{Examples}

\newenvironment{examples*}[1][\alph*]
{
	\refstepcounter{equation}
	\medskip
	\noindent\textbf{Examples.}
	\medskip
	\begin{enumerate}[\bfseries(\theequation.#1),ref=(\theequation.#1),itemsep=5pt]
}
{
	\end{enumerate}
	\smallskip
}

\theoremstyle{remark}

\declaretheoremstyle[
spaceabove=\topsep, 
spacebelow=6pt,
headfont=\normalfont\itshape,
notefont=\normalfont, notebraces={(}{)},
bodyfont=\normalfont,
postheadspace=4pt,
qed=\mbox{\smaller[4]$\boxtimes$}
]{claimproofstyle}

\crefname{subsection}{Subsection}{Subsections}

\crefformat{footnote}{#2\footnotemark[#1]#3}

\crefrangeformat{enumi}{#3#1#4--#5#2#6}

\theoremstyle{plain}

\usepackage{subfiles}

\usepackage{mdframed}
\newmdenv[
leftmargin = 1cm,
rightmargin = 0pt,
skipabove = 8pt,
skipbelow = 3pt,
innerleftmargin = 8pt,
innertopmargin = 0pt,
innerbottommargin = 0pt,
innerrightmargin = 0pt,
linewidth = 3pt,
topline = false,
rightline = false,
bottomline = false
]{leftbar}

\definecolor{gris}{RGB}{90,90,90}

\definecolor{vert}{RGB}{7,126,26}

\definecolor{purple}{RGB}{116,0,159}

\makeatletter
\def\@settitle{\begin{center}%
		\baselineskip14\p@\relax
		\bfseries
		\uppercasenonmath\@title
		\@title
		\ifx\@subtitle\@empty\else
		\\[1ex]\uppercasenonmath\@subtitle
		\footnotesize\mdseries\@subtitle
		\fi
	\end{center}%
}
\def\subtitle#1{\gdef\@subtitle{#1}}
\def\@subtitle{}
\makeatother

\title{A descriptive set theorist's proof of the pointwise ergodic theorem}
\author{Anush Tserunyan}
\address[Anush Tserunyan]{Department of Mathematics, University of Illinois at Urbana-Champaign, IL, 61801, USA}
\email{anush@illinois.edu}
\thanks{The author's research was partially supported by NSF Grant DMS-1501036.}

\makeatletter

\newcommand{\Lone}[1]{\|#1\|_{_1}}
\newcommand{\fbar}{\bar{f}}
\newcommand{\tS}{\tilde{S}}

\newcommand{\meanf}{A_f}

\newcommand{\supf}{\overline{f}}
\newcommand{\inff}{\underline{f}}

\begin{document}

\begin{abstract}
	We give a short combinatorial proof of the classical pointwise ergodic theorem for probability measure preserving $\Z$-actions \cite{Birkhoff:ergodic_theorem}. Our approach reduces the theorem to a tiling problem: tightly tile each orbit by intervals with desired averages. This tiling problem is easy to solve for $\Z$ with intervals as tiles. However, it would be interesting to find other classes of groups and sequences of tiles for which this can be done, since then our approach would yield a pointwise ergodic theorem for such classes.
\end{abstract}

\
\vspace{-1em}
\maketitle

Let $(X,\mu)$ be a standard probability space and $f \in L^1(X,\mu)$. For a finite nonempty $U \subseteq X$, put
\[
\meanf[U] \defeq \frac{1}{|U|} \sum_{y \in U} f(y),
\]
and for a finite equivalence relation $F$ on $X$, define $\meanf[F] : X \to \R$ by $\meanf[F](x) \defeq \meanf\clintbig{[x]_F}$.

\begin{lemma}[Finite averages]\label{fin_mean}
	For any measure-preserving finite equivalence relation $F$ on $(X,\mu)$, 
	\[
	\int f d\mu = \int \meanf[F] \, d \mu.
	\]
\end{lemma}
\begin{proof}
	For each $n \in \N$, restricting to the part of $X$ where each $F$-class has size $n$, we may assume $X$ is that part to begin with. Because each $F$-class is finite, there is a Borel automorphism\footnote{By the Luzin--Novikov countable section uniformization theorem \cite{bible}*{18.10}.\label{footnote:Luzin-Novikov}} $T$ that induces $F$ and a Borel $F$-transversal\footnote{A set that meets every $F$-class at exactly one point.} $B \subseteq X$. Using the invariance of $\mu$, we deduce
	\[
		\displaystyle
		\int_X f(x) d \mu(x)
		= 
		\int_B \sum_{i < n} f(T^i x) d \mu(x)
		=
		\int_B \sum_{i < n} \meanf[F](T^i x) \, d \mu(x)
		=
		\int_X \meanf[F](x) \, d \mu(x)
		.
		\qedhere
	\]
\end{proof}

Let $T$ be an aperiodic automorphism of $(X,\mu)$ and let $\le_T$ denote the induced partial order on $X$, i.e.
$
x \le_T y \defequiv \exists n \in \N \; T^n x = y.
$
For $x, y \in X$, put
$
(x, y)_T \defeq \set{z \in X : x <_T z <_T y}
$
and call the sets of this form \emph{$T$-intervals}; also, define $[x,y)_T$ and $(x,y]_T$ expectedly. We say that subset $S$ of a $T$-orbit is \emph{bi-infinite} if it has not minimum or maximum with respect to $\le_T$. Furthermore, we say that $S$ has a gap bigger than $L \in \N$ if there are $x,y \in S$ with $(x,y)_T \cap S = \0$ and $|[x,y)_T| \ge L$. For $n \in \N$, put $\meanf[T,n](x) \defeq \meanf\clintbig{[x, T^n x)_T}$.

\medskip

We first prove the classical pointwise ergodic theorem for ergodic actions to convey the main idea and then prove the general version with conditional expectation afterwards.

\subsection{For ergodic $\Z$-actions}

\begin{lemma}[Complete sections with large gaps]\label{complete-sections_with_large_gaps}
	If $T$ is ergodic, then for each $L \in \N$, there is a Borel set $S$ of arbitrarily small measure that is disjoint from $\bigcup_{i = 1}^{L} T^{-i} S$ yet meets a.e. orbit in a bi-infinite set.
\end{lemma}
\begin{proof}
	Let $S_0$ be a Borel set of positive measure less than $\frac{1}{L}$, which hence meets a.e. orbit, by ergodicity. Because $\bigcup_{i = 1}^{L} T^{-i} S_0$ has measure less than $1$, $S_0$ must have gaps bigger than $L$ in a positive measure set of orbits, which ergodicity again turns into a.e. orbit. Therefore, $S \defeq S_0 \setminus \bigcup_{i = 1}^{L} T^{-i} S_0$ still meets a.e. orbit and is thus of positive measure. The part of $X$ where $S$ is not bi-infinite is null because we can choose a point in each of those orbits in a Borel fashion\cref{footnote:Luzin-Novikov}.
\end{proof}

\begin{theorem}[Pointwise ergodic for ergodic actions]
	For ergodic $T$, $\displaystyle \lim_{n \to \w} \meanf[T,n](x) = \int f \, d\mu$ a.e.
\end{theorem}
\begin{proof}
	Replacing $f$ with $f - \int f d\mu$, we may assume that $\int f d\mu = 0$. We show that $\bar{f} \defeq \displaystyle\limsup_{n \to \w} \meanf[T,n] \le 0$ a.e. and an analogous argument shows that $\displaystyle \liminf_{n \to \w} \meanf[T,n] \ge 0$ a.e. 
	
	Because $\fbar$ is $T$-invariant, ergodicity implies that it is some constant $c$ a.e. Suppose towards a contradiction that $c > 0$ and put $\de \defeq \frac{c}{2}$. Take $\e \defeq \frac{\min(\de,1)}{8}$ and let $L$ be large enough so that the set 
	\[
	Z \defeq \setbig{x \in X : \meanf[T,n](x) < \de \; \text{for all integers $n \in [1,L]$}}
	\] 
	supports less than $\e$ of the total masses of $1$ and $f$, i.e. $\Lone{\1_Z} + \Lone{f \cdot \1_Z} < \e$. Define a function $\ell : X \to \N$ by mapping $x$ to the largest $n \le L$ such that $\meanf[T,n](x) \ge \de$, if $x \notin Z$, and to $1$, otherwise. For each $x \in X$, put $I_x \defeq [x, T^{\ell(x)} x)_T$, and say that a $T$-interval $I \defeq [y,z)_T$ is \emph{tiled} if it admits a partition (\emph{tiling}) into $T$-intervals of the form $I_x$. It follows by induction on the length of $I$ that such a partition is unique because $I_y$ has to be the tile containing $y$.
	
	Let $S$ be given by \cref{complete-sections_with_large_gaps} applied to $L$. Because $T$ is measure-preserving, we can take $S$ small enough so that the set $\tS \defeq \bigcup_{i = 1}^{L} T^{-i} S$ supports less than $\e$ of the total mass of $1$ and $f$, i.e. $\Lone{\1_{\tS}} + \Lone{f \cdot \1_{\tS}} < \e$. For each $x \in X$, denote by $s(x)$ the closest element of $S$ to the left of $x$, i.e. $s(x) \in S$, $s(x) \le_T x$, and $\big(s(x), x\big]_T \cap S = \0$. Define a partial finite equivalence relation $F$ on $X$ as follows:
	\[
	x F y \defequiv \exists z \notin \tS \text{ such that } x,y \in I_z \text{ and } \big[s(z), z\big)_T \text{ is tiled}.
	\]
	It is clear that $Y \defeq \dom(F) \supseteq X \setminus (\tS \cup Z)$, so $\int_{X \setminus Y} d \mu + \int_{X \setminus Y} |f| d \mu < 2\e$. Also, for each $y \in Y$, $[y]_F = I_z$ for some $z \in X \setminus Z$, so $\meanf[F](y) \ge \de$. Thus, \cref{fin_mean} implies a contradiction:
	\[
	0 = \int_X f d\mu \ge \int_Y f d\mu - 2\e = \int_Y \meanf[F] d\mu - 2\e \ge \de (1 - 2\e) - 2\e \ge \frac{\de}{2} > 0.
	\qedhere
	\]
\end{proof}

\subsection{For general $\Z$-actions}

Now let $T$ be any measurable preserving automorphism. For a set $S \subseteq X$, denote by $[S]_T$ its \emph{$T$-saturation}, i.e. $[S]_T \defeq \bigcup_{n \in \Z} T^n S$.

\begin{lemma}[Approximate complete sections with large gaps]\label{a.complete-sections_with_large_gaps}
	If $T$ is aperiodic, then for each $\e > 0$ and $L \in \N$, there is a Borel set $S$ of arbitrarily small measure that is disjoint from $\bigcup_{i = 1}^{L} T^{-i} S$ yet $[S]_T$ has measure $1 - \e$.
\end{lemma}
\begin{proof}
	By the marker lemma \cite{Kechris-Miller}*{Lemma 6.7}, there is a Borel complete $T$-section $S_0$ of arbitrarily small measure, in particular, smaller than $\frac{\e}{L}$. Because $\bigcup_{i = 1}^{L} T^{-i} S_0$ has measure less than $\e$, $S_0$ must have gaps bigger than $L$ in a set of orbits of measure larger than $1 - \e$. Therefore, the $T$-saturation of $S \defeq S_0 \setminus \bigcup_{i = 1}^{L} T^{-i} S_0$ has measure larger than $1-\e$.
\end{proof}

\begin{theorem}[Pointwise ergodic for general actions]
	$\displaystyle \lim_{n \to \w} \meanf[T,n](x)$ exists a.e. and is equal to the conditional expectation of $f$ with respect to the $\si$-algebra of $T$-invariant measurable sets.
\end{theorem}
\begin{proof}
	The part about the condition expectation of $f$ follows by a standard argument from the first part and the $T$-invariance of $\mu$. To this end, we first argue that $f_0 \defeq \displaystyle \lim_{n \to \w} \meanf[T,n](x)$ is integrable by Fatou's lemma and the $T$-invariance of $\mu$:
	\[
	\Lone{f_0} \le \liminf_{n \to \w} \int_X \frac{1}{n} \sum_{i = 0}^{n-1} |f(T^i x)| d\mu(x) = \liminf_{n \to \w} \frac{1}{n} \sum_{i = 0}^{n-1} \int_X |f(T^i x)| d\mu = \int_X |f| d\mu.
	\]
	Thus, for any $T$-invariant measurable set $A \subseteq X$, by the generalized dominated convergence theorem:
	\[
	\int_A f_0 d\mu = \liminf_{n \to \w} \frac{1}{n} \sum_{i = 0}^{n-1} \int_A f(T^i x) d\mu(x) = \int_A f d\mu.
	\]
	
	We now turn to the a.e. existence of the limit. This easily follows from \cref{fin_mean} for the part of $X$ where $T$ is periodic, so we assume that $T$ is aperiodic and show that 
	\[
	\supf \defeq \displaystyle\limsup_{n \to \w} \meanf[T,n] \le \displaystyle \liminf_{n \to \w} \meanf[T,n] \eqdef \inff \text{ a.e.}
	\]
	Suppose towards a contradiction that there are $a < b \in \R$ such that the set 
	\[
	X' \defeq \set{x \in X : \inff(x) < a < b < \supf(x)}
	\] 
	has positive measure. Because $X'$ is $T$-invariant, we assume, as we may, that $X' = X$.
	
	Fixing $\e > 0$ such that $(b - a) \mu(X) > 2 \e (|a| + |b| + 2)$, we first focus on $\supf$ and $b$. 
	
	Let $L$ be large enough so that the set 
	\[
	Z \defeq \set{x \in X : (\forall n, 1 \le n \le L)\; \meanf[T,n](x) < b}
	\] 
	supports less than $\e$ of the total masses of $1$ and $f$, i.e. $\Lone{\1_Z} + \Lone{f \cdot \1_Z} < \e$. Define a function $\ell : X \to \N$ by mapping $x$ to the largest $n \le L$ such that $\meanf[T,n](x) \ge b$, if $x \notin Z$, and to $1$, otherwise. For each $x \in X$, put $I_x \defeq [x, T^{\ell(x)} x)_T$, and say that a $T$-interval $I \defeq [y,z)_T$ is \emph{tiled} if it admits a partition (\emph{tiling}) into $T$-intervals of the form $I_x$. It follows by induction on the length of $I$ that such a partition is unique because $I_y$ has to be the tile containing $y$.
	
	Let $S$ be given by \cref{a.complete-sections_with_large_gaps} applied to $L$. Because $T$ is measure-preserving, we can take $S$ small enough so that the union of the sets $[S]_T^c$ and $\tS \defeq \bigcup_{i = 1}^{L} T^{-i} S$ supports less than $\e$ of the total mass of $1$ and $f$, i.e. $\Lone{\1_{[S]_T^c \cup \tS}} + \Lone{f \cdot \1_{[S]_T^c \cup \tS}} < \e$. For each $x \in X$, denote by $s(x)$ the closest element of $S$ to the left of $x$, i.e. $s(x) \in S$, $s(x) \le_T x$, and $\big(s(x), x\big]_T \cap S = \0$.
	
	Define a partial finite equivalence relation $F$ on $X$ as follows:
	\[
	x F y \defequiv \exists z \notin \tS \text{ such that } x,y \in I_z \text{ and } \big[s(z), z\big)_T \text{ is tiled}.
	\]
	It is clear that $Y \defeq \dom(F) \supseteq X \setminus ([S]_T^c \cup \tS \cup Z)$, so $\int_{X \setminus Y} d \mu + \int_{X \setminus Y} |f| d \mu < 2\e$. Also, for each $y \in Y$, $[y]_F = I_z$ for some $z \in X \setminus Z$, so $\meanf[F](y) \ge b$. Thus, using \cref{fin_mean}, we get
	\[
	\int_X f d\mu \ge \int_Y f d\mu - 2\e = \int_Y \meanf[F] d\mu - 2\e \ge b \mu(Y) - 2\e \ge b \mu(X) - 2\e (|b| + 1).
	\]
	
	An analogous argument for $\inff$ and $a$ gives $\int_X f d \mu \le a \mu(X) + 2 \e (|a| + 1)$, so $(b - a) \mu(X) \le 2 \e (|a| + |b| + 2)$, contradicting the choice of $\e$.
\end{proof}

\begin{remark}
	Another short proof of the pointwise ergodic theorem for $\Z$ is given by Keane and Petersen in \cite{Keane-Petersen:ergodic_thm}. The proof is analytic and has the advantage of not using any black box, whereas we do use the Luzin--Novikov uniformization theorem to keep the sets measurable\footnote{Although, instead, we could easily observe that these sets are in the $\si$-ideal generated by analytic sets and use that these sets are universally measurable.}. The tiling is implicitly present in Keane--Petersen proof, but without turning it into an equivalence relation, so it is not clear how to adapt their proof to other shapes of tiles in other groups.
		
	Our approach explicitly reduces the pointwise ergodic theorem to a tiling problem, which makes it interesting to consider this problem for other groups and sequences of tiles. If solved, our approach then would yield a pointwise ergodic theorem for those groups.
\end{remark}

%%%%%%%%%%%%%%%%%%%%%%%%%%%%%%%%%%%%%%%%%%%%%%%%%%
%%%%%%%%%%%%%%%%%%%%%%%%%%%%%%%%%%%%%%%%%%%%%%%%%%

\bigskip

%\newpage
%\newsavebox\mytempbib \savebox\mytempbib{\parbox{\textwidth}{

%}}

\end{document}